\documentclass[12pt]{amsart}
\usepackage{amssymb}
\usepackage{amstext}
\usepackage{amsmath}
\usepackage{amscd}
\usepackage{latexsym}
\usepackage{amsfonts}

\theoremstyle{plain}
\newtheorem{thm}{Theorem}[section]
\newtheorem*{thm*}{Theorem}
\newtheorem*{cor*}{Corollary}

\newtheorem{prop}[thm]{Proposition}
\newtheorem{lem}[thm]{Lemma}
\newtheorem{cor}[thm]{Corollary}

\newtheorem*{claim*}{Claim}

\theoremstyle{definition}
\newtheorem{defn}[thm]{Definition}
\newtheorem{ex}[thm]{Example}
\newtheorem{rem}[thm]{Remark}

\newtheorem{conj}[thm]{Conjecture}

\theoremstyle{remark}

\numberwithin{equation}{thm}

\def\Ker{\mathrm{Ker}}

\def\a{\mathfrak a}

\def\m{\mathfrak m}
\def\n{\mathfrak n}

\def\p{\mathfrak p}
\def\P{\mathfrak P}
\def\q{\mathfrak q}

\def\P{\mathfrak P}

\def\H{\mathrm{H}}

\newcommand{\calF}{\mathcal{F}}

\newcommand{\fkm}{\mathfrak{m}}

\newcommand{\fkp}{\mathfrak{p}}

\def\depth{\mathrm{depth}}

\def\Ass{\mathrm{Ass}}
\def\Assh{\mathrm{Assh}}
\def\Min{\mathrm{Min}}

\def\Spec{\mathrm{Spec}}

\newcommand{\rar}{\rightarrow}
\newcommand{\lar}{\longrightarrow}

\def\cl{\overline}
\def\ds{\displaystyle}
\def\depth{\mbox{\rm depth}}

\begin{document}

\setlength{\baselineskip}{24pt}
\title{Cohen--Macaulayness versus the vanishing of the first Hilbert
coefficient of parameter ideals}
\pagestyle{plain}
\author{L. Ghezzi}
\address{Department of Mathematics, New York City College of Technology-Cuny, 300 Jay Street, Brooklyn, NY 11201, U. S. A.}
\email{lghezzi@citytech.cuny.edu}
\author{S. Goto}
\address{Department of Mathematics, School of Science and Technology, Meiji University, 1-1-1 Higashi-mita, Tama-ku, Kawasaki 214-8571, Japan}
\email{goto@math.meiji.ac.jp}
\author{J. Hong}
\address{Department of Mathematics, Southern Connecticut State University, 501 Crescent Street, New Haven, CT 06515-1533, U. S. A.}
\email{hongj2@southernct.edu}
\author{K. Ozeki}
\address{Department of Mathematics, School of Science and Technology, Meiji University, 1-1-1 Higashi-mita, Tama-ku, Kawasaki 214-8571, Japan}
\email{kozeki@math.meiji.ac.jp}
\author{T.T. Phuong}
\address{Department of Information Technology and Applied Mathematics,
Ton Duc Thang University, 98 Ngo Tat To Street, Ward 19, Binh Thanh District,
Ho Chi Minh City, Vietnam}
\email{sugarphuong@gmail.com}
\author{W.V. Vasconcelos}
\address{Department of Mathematics, Rutgers University, 110 Frelinghuysen Rd, Piscataway, NJ 08854-8019, U. S. A.}
\email{vasconce@math.rutgers.edu}

\thanks{{AMS 2000 {\em Mathematics Subject Classification:}
13H10, 13H15, 13A30.}\\The first author is partially supported by a grant from the City University of New York PSC-CUNY Research Award Program-40. The second author is partially supported by Grant-in-Aid for Scientific Researches (C) in Japan (19540054).
The fourth author is supported by a grant from MIMS (Meiji Institute for Advanced Study of Mathematical Sciences).
The fifth author is supported by JSPS Ronpaku (Dissertation of PhD) Program.
The last author is partially supported by the NSF}

\begin{abstract}
The conjecture of Wolmer  Vasconcelos \cite{V} on the vanishing of the
first Hilbert coefficient $e_1(Q)$ is solved affirmatively, where
$Q$ is a parameter ideal in a Noetherian local ring.
Basic properties of the rings for which $e_1(Q)$ vanishes
 are derived. The
invariance of $e_1(Q)$ for parameter ideals $Q$ and its relationship
to Buchsbaum rings are studied.
\end{abstract}

\maketitle

\section{Introduction}
Let $A$ be a Noetherian local ring with maximal ideal $\m$ and $d
= \operatorname{dim} A > 0$. Let $\ell_A(M)$ denote, for an
$A$-module $M$, the length of $M$. Then, for each $\m$-primary ideal
$I$ in $A$, we have integers $\{e_i(I)\}_{0 \le i \le d}$ such that
the equality
$$\ell_A(A/I^{n+1})={e}_0(I)\binom{n+d}{d}-
e_1(I)\binom{n+d-1}{d-1}+\cdots+(-1)^d{e}_d(I)$$
holds true for all integers $n \gg 0$. We call $\{e_i(I)\}_{0 \le i \le d}$ the Hilbert
coefficients of $A$ with respect to $I$.
These integers carry a great deal of information about the ideal $I$.
We will argue that $e_1(Q)$, for parameter ideals $Q$, codes
structural information about the ring $A$ itself. Noteworthy properties of $A$ associated to values of $e_1(Q)$ are the Cohen--Macaulay, the generalized Cohen--Macaulay and the Buchsbaum conditions.

 We say that $A$ is unmixed,
if $\operatorname{dim} \widehat{A}/\p = d$ for every $\p \in
\operatorname{Ass} (\widehat{A})$, where $\widehat{A}$ denotes the
$\m$-adic completion of $A$. With this notation Vasconcelos, exploring the vanishing of $e_1(Q)$ for parameter ideals
$Q$, posed the
following conjecture in his lecture at the conference in Yokohama in March 2008.

\begin{conj}[\cite{V}]\label{conj}
Assume that $A$ is unmixed. Then $A$ is a Cohen--Macaulay local ring,
once  $e_1(Q) = 0$ for some parameter ideal $Q$ of $A$.
\end{conj}

In Section 2 of the present paper we shall settle Conjecture
\ref{conj} affirmatively (Theorem 2.1). Here we should note that Conjecture
\ref{conj} is already solved partially by \cite{GhHV} and \cite{MSV}.
In fact, Ghezzi, Hong, and Vasconcelos \cite[Theorem 3.3]{GhHV}
proved that the conjecture holds true, if $A$ is an integral domain
which is a homomorphic image of a Cohen--Macaulay ring. Mandal, Singh and
Verma \cite{MSV} proved that $e_1(Q) \le
0$ for every parameter ideal $Q$ in an arbitrary Noetherian local
ring $A$ and showed that $e_1(Q) < 0$, if $\operatorname{depth} A
= d-1$.

\medskip

 \noindent {\bf Theorem 2.1} \;
  Let $A$ be a Noetherian local ring with $d = \dim A > 0$ and let $Q$ be a parameter ideal in $A$. Then following are equivalent{\rm :}
\begin{enumerate}
\item[{\rm (a)}] $A$ is Cohen--Macaulay{\rm ;}
\item[{\rm (b)}] $A$ is {unmixed} and $e_1(Q)=0${\rm ;}
\item[{\rm (c)}] $A$ is {unmixed} and $e_1(Q) \geq 0$.
\end{enumerate}

\medskip

Let $\operatorname{Assh} (A) =\{ \p \in \operatorname{Ass} (A) \mid
\operatorname{dim} A/\p = d\}$ and let $(0) = \bigcap_{\p \in
\operatorname{Ass} (A)}I(\p)$ be a primary decomposition of
$(0)$ in $A$ with $\p$-primary ideals $I(\p)$ in $A$. We put
$$U_A(0) = \bigcap_{\p \in \operatorname{Assh} (A)}
I(\p)$$ and call it the {\it unmixed component} of $(0)$ in $A$.

Let us call those local rings $A$ with $e_1(Q) = 0$ for some
parameter ideal $Q$ of $A$ {\it Vasconcelos}\footnote{The terminology
is due to the first five authors} rings. In Section 3 we shall explore
basic properties of Vasconcelos rings. Certain sequentially
Cohen--Macaulay rings are good examples of Vasconcelos rings. A basic
characterization of some of these rings is:

\medskip

 \noindent {\bf Theorem 2.7} \; Let $A$ be a Noetherian
local ring of dimension $d \geq 2$. Let $U=U_A(0)$ and $Q$ a parameter
ideal of $A$.  Suppose that $A$ is a homomorphic image of a
Cohen--Macaulay ring. Then the following are equivalent:
\begin{enumerate}
\item[{\rm (a)}] $e_1(Q)=0$;
\item[{\rm (b)}] $A/U$ is Cohen--Macaulay and $\dim U \leq d-2$.
\end{enumerate}

Notice that unless $A$ is a homomorphic image of a Cohen--Macaulay
ring, the implication (a) $\Rightarrow$ (b) is not true in general
(Remark \ref{lech}).

In Section 4 we will study the problem of when $e_1(Q)$ is
 independent of the choice of the parameter ideal $Q$ in
$A$. We shall show that $A$ is a quasi-Buchsbaum ring, if $A$ is
unmixed and $e_1(Q)$ is constant (Corollary \ref{q-Bbm}).  The
authors conjecture that $A$ is furthermore a Buchsbaum ring, if $A$ is
unmixed and $e_1(Q)$ is independent of the choice of parameter
ideals $Q$ of $A$. We will show that this is the case, at least when
$e_1(Q)=-1$ or $-2$ (Theorem~\ref{constant}, Theorem~\ref{constant2}).
Goto and Ozeki \cite{GO} recently solved the conjecture affirmatively.

Another important issue is that of  the variability of $e_1(Q)$,
sometimes for $Q$ in a same integral closure class, and its role in
the structure of the ring. This will be pursued in a sequel paper.

In what follows, unless otherwise specified, let $(A,\m)$ denote a
Noetherian local ring with maximal ideal $\m$ and $d =
\operatorname{dim} A$. Let $\{\H_{\m}^i(*)\}_{i \in \mathbb Z}$ be the
local cohomology functors of $A$ with respect to the maximal ideal
$\m$.

\section{The vanishing conjecture}
The purpose of this section is to prove the following, which settles Conjecture \ref{conj} affirmatively.
Throughout let $A=(A, \m)$ be a Noetherian local ring with maximal ideal $\m$ and $d=\dim A$.

\begin{thm}\label{2.1}  Let $A$ be a Noetherian local ring with $d = \dim A > 0$ and let $Q$ be a parameter ideal in $A$. Then following are equivalent{\rm :}
\begin{enumerate}
\item[{\rm (a)}] $A$ is Cohen--Macaulay{\rm ;}
\item[{\rm (b)}] $A$ is {unmixed} and $e_1(Q)=0${\rm ;}
\item[{\rm (c)}] $A$ is {unmixed} and $e_1(Q) \geq 0$.
\end{enumerate}
\end{thm}

In our proof of Theorem \ref{2.1} the following facts are the key. See {\cite[Section 3]{GNa}} for the proof.

\begin{prop}[\cite{GNa}]\label{GNa}
Let $(A,\fkm)$ be a Noetherian local ring with $d = \dim A \ge 2$, possessing the canonical module $\mathrm{K}_A$. Assume that $\dim A/\fkp = d$ for every $\fkp \in \mathrm{Ass} (A) \setminus \{\fkm \}$. Then the following assertions hold true.
\begin{enumerate}
\item[{\rm (a)}] The local cohomology module $\mathrm{H}_{\fkm}^1(A)$ is finitely generated.
\item[{\rm (b)}] The set $\calF = \{ \fkp \in \Spec A \mid \dim A_{\fkp} > \depth (A_{\fkp}) = 1 \}$ is finite.
\item[{\rm (c)}] Suppose that the residue class field $k=A/\fkm$ of $A$ is infinite and let $I $ be an $\fkm$-primary ideal in $A$. Then one can choose an element $a \in I$ so that $a$ is superficial for $I$ and $\dim A/\fkp = d - 1$ for every $\fkp \in \mathrm{Ass}_A (A/aA) \setminus \{\fkm \}$.
\end{enumerate}
\end{prop}

\begin{rem}\label{U(A)}
Let $(A, \m)$ be a Noetherian local ring with $\dim A=d>0$. Recall that the unmixed component of $(0)$ in $A$ is
$U_A(0) = \bigcap_{\p \in \operatorname{Assh} (A)}
I(\p)$. Since $\H^0_{\m}(A)= \bigcap_{\p \in \Ass(A) \setminus \{ \m \}}I(\p)$, we have that $\H^0_{\m} (A) \subseteq U_A(0)$. If $\Ass(A) \setminus \{\m\} = \Assh(A)$,
then $\H^0_{\m} (A) = U_A(0)$.
\end{rem}

\begin{proof}[Proof of Theorem $\ref{2.1}$] (a) $\Rightarrow$ (b) $\Rightarrow$ (c) are clear. In order to show (c) $\Rightarrow$ (a), we may assume that $A$ is a complete unmixed local ring with $d\ge 2$ and infinite residue field. Let $Q=(a_1, \ldots, a_d)$.
We use induction on $d$.


\medskip

\noindent Let $d=2$. Then $Q=(a_1, a_2)$, where we may assume that $a=a_1$ is a superficial element.
Let $S=A/aA$ and let $q = QS$. Then $\dim S=1$ and by \cite[Lemma 2.2]{GNi} we have
\[ - \ell_A(\H^0_{\m}(S)) = e_1(q) = e_1(Q)- \ell_A(0:_A a )= e_1(Q) \geq 0.   \]
Hence $\H^0_{\m}(S)=(0)$. Therefore $S$ is Cohen--Macaulay and so is $A$.

\medskip

\noindent Suppose $d \geq 3$. Then there exists $a \in Q$ such that $\Ass(A/aA) \subseteq  \Assh(A/aA) \cup \{ \m \}$ (Proposition~\ref{GNa} (c)). Let $S=A/aA$ and $q=QS$. Note that $S$ is not necessarily unmixed. Let $U=U_S(0)$,
$\cl{S}=S/U$ and $\cl{q}=q \cl{S}$. Then $\cl{S}$ is unmixed of dimension $d-1$.
Since $e_1(\cl{q})=e_1(q)=e_1(Q) \geq 0$, by the induction hypothesis $\cl{S}$ is Cohen--Macaulay, i.e., $\H_{\m}^{i}(\cl{S}) =(0)$ for all $0 \leq i \leq d-2$.

\smallskip

\noindent From the exact sequence $0 \rar U \rar S \rar \cl{S} \rar 0$, we get a long exact sequence
\[ \begin{array}{lll}
 0 &\lar & \H^{0}_{\m}(U) \lar  \H^{0}_{\m}(S)  \lar  \H^{0}_{\m}(\cl{S}) \lar  \H^{1}_{\m}(U) \lar  \H^{1}_{\m}(S) \lar  \H^{1}_{\m}(\cl{S}) \lar \cdots \\ && \\
   &\lar & \H^{d-2}_{\m}(U)  \lar  \H^{d-2}_{\m}(S) \lar  \H^{d-2}_{\m}(\cl{S}) .
 \end{array}\]
Since $\H_{\m}^{0}(S)=U$ (Remark~\ref{U(A)}),  we have $\H^{i}_{\m}(U)=(0)$ for all $i \geq 1$. Therefore
\[\H^{i}_{\m}(S)=(0), \quad \mbox{\rm for all} \;\; 1 \leq i \leq d-2. \]

\smallskip

\noindent From the exact sequence $0 \rar A \stackrel{\cdot a}{\lar} A \rar S \rar 0$, we get
\[\begin{array}{lll}
0 &\lar & \H^{0}_{\m}(A) \lar  \H^{0}_{\m}(A)  \lar  \H^{0}_{\m}(S) \lar  \H^{1}_{\m}(A)  \stackrel{\cdot a}{\lar}
\H^{1}_{\m}(A) \lar  0 \lar \cdots \\ && \\
 &\lar & 0 \rar \H^{i}_{\m}(A) \stackrel{\cdot a}{\lar}  \H^{i}_{\m}(A) \rar 0 \rar \cdots \rar  \H^{d-2}_{\m}(S)=0 \rar \H^{d-1}_{\m}(A)  \stackrel{\cdot a}{\lar} \H^{d-1}_{\m}(A)  .
 \end{array}\]

\smallskip
\noindent Since $A$ has a nonzero divisor, 
$\H^{0}_{\m}(A)=(0)$.
The epimorphism $\H^{1}_{\m}(A)  \rar \H^{1}_{\m}(A) \rar  0$ implies $\H^{1}_{\m}(A) =a \H^{1}_{\m}(A) $. Since $\H^{1}_{\m}(A) $ is finitely generated (Proposition~\ref{GNa} (a)), $\H^{1}_{\m}(A) =(0)$.
For $2 \leq i \leq d-1$, we obtain
 $\H^{i}_{\m}(A) =(0)$ because for every $x \in \H^{i}_{\m}(A)$ some power of $a$ annihilates $x$.
\end{proof}

Let us discuss some consequences of Theorem \ref{2.1}.

\begin{lem}\label{lemma}
Let $A$ be a Noetherian local ring of dimension $d > 0$. Let $Q$ be a parameter ideal of $A$.
Suppose that $U=U_A(0) \neq (0)$.
Let $C = A/U$. Then the following assertions hold true.
\begin{enumerate}
\item[{\rm (a)}] $\dim U < \dim A$.
\item[{\rm (b)}] We have
\[
e_1(Q) = \left\{\begin{array}{ll}
e_1(QC) &\quad \mbox{\rm if} \;\; \dim U \leq d-2 \\& \\
e_1(QC)-s_0 &\quad \mbox{\rm if} \;\; \dim U=d-1,
\end{array}
\right.
\]where $s_0 \geq 1$ is the multiplicity of  ${\ds \bigoplus_n U/(Q^{n+1}\cap U)}$.

\item[{\rm (c)}] $e_1(Q) \leq e_1(QC)$ with equality if and only if $\dim U \leq d-2$.
\end{enumerate}
\end{lem}

\begin{proof} (b) We write
$$\ell_A(U/(Q^{n+1}\cap U)) = s_0\binom{n + t }{t} -s_1\binom{n + t-1}{t-1} + \cdots + (-1)^ts_t$$
for $n \gg 0$ with integers $\{s_i\}_{0 \le i \le t}$, where $t=\dim U$.
Then the claim follows from
the exact sequence $0 \rar U \rar A \rar C \rar 0$  of $A$--modules, which gives
\[ \ell_A(A/Q^{n+1}) = \ell_A(C/Q^{n+1}C) + \ell_A(U/(Q^{n+1}\cap U)) \]
for all $n \ge 0$.

\medskip

\noindent (c) It follows from (b) and the fact that $s_0 \ge 1$.
\end{proof}

The following results are due to \cite{MSV}. We include an independent proof.

\begin{cor}[{\cite{MSV}}]\label{cor} Let $A$ be a Noetherian local ring of dimension $d >0$. Let $Q$ be a parameter ideal of $A$. Then the following assertions hold true.
\begin{enumerate}
\item[{\rm (a)}] $e_1(Q) \le 0$.
\item[{\rm (b)}] If $\depth(A) =d-1$, then $e_1(Q) < 0$.
\end{enumerate}
\end{cor}

\begin{proof}
 (a) We may assume that $A$ is complete. Let $U=U_A(0)$ and $C=A/U$. Then by Lemma~\ref{lemma},
  $e_1(Q) \le e_1(QC)$. Hence we may also assume that $A$ is unmixed. Hence  the claim follows from Theorem \ref{2.1}.

\medskip

\noindent (b) We may assume that the residue field $A/\m$ is infinite. If $d=1$, by \cite[Lemma 2.4 (1)]{GNi}, we have
$e_1(Q)= - \ell_A(H^0_{\m}(A)) < 0$,
where the last inequality follows from the fact that $\depth(A)=0$.
Suppose that $d \geq 2$. Let $a_1,\dots, a_{d-1}\in Q$ be a superficial sequence and let $\cl{A}=A/(a_1,\dots, a_{d-1})$, $\cl{Q}=Q/(a_1,\dots, a_{d-1})$. Since $\depth(A)=d-1$, we have that
$e_1(Q)= e_1(\cl{Q})=- \ell_A(H^0_{\m}(\cl{A})) < 0$.
\end{proof}

\begin{prop}\label{1.2-1}
Let $A$ be a Noetherian local ring of dimension $d \geq 2$. Let $U=U_A(0)$.
Suppose that $A/U$ is Cohen--Macaulay and $\dim U \leq d-2$. Then
$e_1(Q)=0$ for every parameter ideal $Q$ of $A$.
\end{prop}

\begin{proof} Let $C=A/U$. Since $\dim U \leq d-2$, we get  $e_1(Q)=e_1(QC)$ (Lemma~\ref{lemma}). Since $C$ is Cohen--Macaulay, we have $e_1(QC)=0$ (Theorem~\ref{2.1}), which completes the claim.
\end{proof}

\begin{thm}\label{1.2-2}
Let $A$ be a Noetherian local ring of dimension $d \geq 2$. Suppose that $A$ is a homomorphic image of a Cohen--Macaulay ring. Let $U=U_A(0)$ and let $Q$ be a parameter ideal of $A$.
Then the following are equivalent{\rm :}
\begin{enumerate}
\item[{\rm (a)}] $e_1(Q)=0${\rm ;}
\item[{\rm (b)}] $A/U$ is Cohen--Macaulay and $\dim U \leq d-2$.
\end{enumerate}
\end{thm}

\begin{proof}
 It is enough to prove that (a) $\Rightarrow$ (b).  Let $C = A/U$. Then $C$ is an unmixed local ring because $C$ is a homomorphic image of a Cohen--Macaulay ring  and  $\operatorname{dim} C/P = d$ for all $P \in \operatorname{Ass} (C)$. If $U =(0)$, then $A$ is unmixed. Therefore $A$ is Cohen--Macaulay by Theorem~\ref{2.1}.
 Suppose that $U \ne (0)$. We show that $\dim U \leq d-2$. Suppose not, i.e., $\dim U = d-1$. By Lemma~\ref{lemma}, we get
  \[ e_1(Q) = e_1(QC) - s_0, \]where $s_0 \geq 1$. On the other hand, $e_1(QC) \leq 0$ by Corollary~\ref{cor} (a). Hence $e_1(Q) <0$, which is a contradiction. Therefore $\dim U \leq d-2$. Now Lemma~\ref{lemma} implies that  $e_1(QC) = e_1(Q) = 0$. By Theorem~\ref{2.1}, $C$ is a Cohen--Macaulay ring.
\end{proof}

The following corollary gives a characterization of Cohen-Macaulayness.

\begin{cor}
Let $A$ be a Noetherian local ring of dimension $d > 0$. Let $Q$ be a parameter ideal in $A$. Suppose that $e_i(Q) = 0$ for all $1 \le i \le  d$. Then $A$ is a Cohen--Macaulay ring.
\end{cor}

\begin{proof}
We may assume that $A$ is complete. Let $U = U_A(0)$. By Theorem~\ref{2.1} it is enough to show that $U=(0)$.
Suppose that $U \neq (0)$ and let $C=A/U$. Since $e_1(Q)=0$ we have that $C$ is Cohen--Macaulay and $\dim U \leq d-2$ (Theorem \ref{1.2-2}).
From the exact sequence $0 \rar U \rar A \rar C \rar 0$  of $A$-modules,
we get
\[ \ell_A(A/Q^{n+1}) = \ell_A(C/Q^{n+1}C) + \ell_A(U/(Q^{n+1}\cap U)) \]
for all $n \ge 0$. By assumption, for $n \gg 0$ we have
\[ \ell_A(A/Q^{n+1})  = e_0(Q) \binom{n+d}{d}.\] Also since $C$ is Cohen--Macaulay, for $n \gg 0$ we have
\[  \ell_A(C/Q^{n+1}C) = e_0(QC) \binom{n+d}{d}.\]  Since $e_0(Q)=e_0(QC)$, for $n \gg 0$ we obtain
\[\ell_A(U/(Q^{n+1}\cap U))=0,\] i.e., $U \subseteq Q^{n+1}$ for $n \gg 0$.
Thus $U = (0)$, which is a contradiction. 
 \end{proof}

The implication (a) $\Rightarrow$ (b) in Theorem \ref{1.2-2} is not true in general without the assumption  that $A$ is a homomorphic image of a Cohen--Macaulay ring.

\begin{rem}\label{lech}
Let $R$ be a Noetherian equi-characteristic complete  local ring and assume that $\operatorname{depth} (R) >0$.
Then one can construct a Noetherian local integral domain $(A,\m)$ such that $R = \widehat{A}$, where $\widehat{A}$ denotes the $\m$-adic completion of $A$ (\cite{L}). For example,  let $k[[X,Y,Z,W]]$ be  the formal power series ring over a field $k$ and consider the local ring $$R = k[[X, Y, Z, W]]/(X) \cap (Y,Z,W).$$ We can choose a Noetherian local integral domain $(A,\m)$ so that $R = \widehat{A}$. Then $e_1(Q)= 0$ for every parameter ideal $Q$ in $A$, since $e_1(Q) = e_1(QR) = 0$. But $A$ is not Cohen--Macaulay because $R= \widehat{A}$ is not Cohen--Macaulay.
\end{rem}

\section{Vasconcelos rings}

Throughout this section let $A=(A, \m)$ be a Noetherian local ring with maximal ideal $\m$ and $d = \dim A$.
Our purpose is to  develop a theory of Vasconcelos rings.  Let us begin with the definition.

\begin{defn}
A Noetherian local ring $A$ is a {\em Vasconcelos} ring, if either $\dim A = 0$, or $\dim A > 0$ and $e_1(Q) = 0$ for some parameter ideal $Q$ in $A$.
\end{defn}

A Cohen--Macaulay local ring is a Vasconcelos ring.

\begin{prop}\label{Vring} Let $(A, \m)$ be a Noetherian local ring with $\dim A =d$.
\begin{enumerate}
\item[{\rm (a)}] A $1$--dimensional Vasconcelos ring is Cohen--Macaulay.
\item[{\rm (b)}] An unmixed Vasconcelos ring is Cohen--Macaulay.
\item[{\rm (c)}] Suppose that $d \geq 2$. Then $A$ is a Vasconcelos ring if and only if $A/\H^0_{\m}(A)$ is a Vasconcelos ring.
\end{enumerate}
\end{prop}

\begin{proof} (a) There exists a parameter ideal $Q$ of $A$ such that $e_1(Q)=0$. By \cite[2.4 (1)]{GNi}, we have
$0= e_1(Q) = -\ell_A(\H^0_{\m}(A))$,
which shows that $A$ is Cohen--Macaulay.

\noindent (b) It follows from Theorem~\ref{2.1}.

\noindent (c) Let $B = A/\H_\m^0(A)$. Then $e_1(QB) = e_1(Q)$ for every parameter ideal $Q$ in $A$. Hence $A$ is a Vasconcelos ring if and only if  $B$ is a Vasconcelos ring.
\end{proof}

Let $A$ be a Noetherian local ring of dimension $d \geq 1$ and let $M$ be a finite $A$--module.
We denote
\[\Assh(M) = \{  \p \in \Ass(M) \mid \dim A/\p = \dim M \}.   \]
Let ${ (0_M) = \bigcap_{\p \in \Ass(M)} M(\p) }$ be a primary decomposition, where $M(\p)$ is $\p$--primary.
The {\em unmixed component} $U_M(0)$ of $(0)$ in $M$ is
\[ U_M(0) = \bigcap_{\p \in \Assh(M)} M(\p).  \]
Note that for any $\p \in \Assh(M)$, we have $U_M(0)_{\p} = (0)$.

\begin{lem}\label{U(M)}
Let $A$ be a Noetherian local ring of dimension $d \geq 1$. Let $M$ be a finite $A$--module and $L$ a submodule of $M$. Suppose that $\dim L \leq \dim M -1$ and $\Ass(M/L) \subseteq \Assh(M)$.
Then $L=U_M(0)$.
\end{lem}

\begin{proof}
Let $\dim M =g$ and $U=U_M(0)$. Since $\dim L \leq g-1$, we have $L_{\p}=(0)$ for every $\p \in \Assh(M)$.
Notice that for each $\p \in \Assh(M)$ there exists $s \in A \setminus \p$ such that $s L =(0)$.
Then $L \subseteq M(\p)$ for every $\p \in \Assh(M)$ because $s$ is a non zero divisor on $M/M(\p)$ and
$M(\p)$ is $\p$--primary.  This means that
\[ L \subseteq \bigcap_{\p \in \Assh(M)} M(\p) = U.\]
Now suppose that $L \subsetneq U$. Then exists $\q \in \Ass(U/L)$. By assumption, $\q \in \Assh(M)$. Hence $L_{\q} =(0)=U_{\q}$, so that $(U/L)_{\q}=(0)$, which is a contradiction.
\end{proof}

Here is a basic characterization of Vasconcelos rings.

\begin{thm}\label{2.5} Let $(A, \m)$ be a Noetherian local ring of dimension $d \geq 2$. Let $\widehat{A}$ be the $\m$-adic completion of $A$.
The following are equivalent{\rm :}
\begin{enumerate}
\item[{\rm (a)}] $A$ is a Vasconcelos ring{\rm ;}
\item[{\rm (b)}] $e_1(Q) = 0$ for every parameter ideal $Q$ in $A${\rm ;}
\item[{\rm (c)}] $\widehat{A}/U_{\widehat{A}}(0)$ is a Cohen--Macaulay ring and $\operatorname{dim}_{\widehat{A}} U_{\widehat{A}}(0) \le d-2${\rm ;}
\item[{\rm (d)}]  There exists a proper ideal $I$ of $\widehat{A}$ such that $\widehat{A}/I$ is a Cohen--Macaulay ring of dimension $d$ and $\operatorname{dim}_{\widehat{A}}I \le d-2$.
\end{enumerate}
When this is the case, $\widehat{A}$ is a Vasconcelos ring, $\H_{\m}^{d-1}(A) = (0)$, and the canonical module $K_{\widehat{A}}$ of $\widehat{A}$ is a Cohen--Macaulay $\widehat{A}$--module.
\end{thm}

\begin{proof} (c) $\Rightarrow$ (b) follows from Proposition~\ref{1.2-1} and (a) $\Rightarrow$ (c) follows from Theorem~\ref{1.2-2}. The implications (b) $\Rightarrow$ (a) and (c) $\Rightarrow$ (d) are trivial.
Finally (d) $\Rightarrow$ (c) follows from Lemma~\ref{U(M)}.

To see the last assertions,
 let $U=U_{\widehat{A}}(0)$ and let $\widehat{\m} = \m \widehat{A}$ be the maximal ideal of $\widehat{A}$. Then $\widehat{A}/U$ is a Cohen--Macaulay ring and $\operatorname{dim}_{\widehat{A}} U \le d-2$. In particular, $\H_{\widehat{\m}}^{d-1}(\widehat{A}/U) =(0) = \H_{\widehat{\m}}^{d-1}(U)$. Hence $\H_{\widehat{\m}}^{d-1}(\widehat{A}) = (0)$. Moreover we have $\H_{\widehat{\m}}^{d}(\widehat{A}) \cong \H_{\widehat{\m}}^{d}(\widehat{A}/U)$, which means that $\mathrm{K}_{\widehat{A}} \cong \mathrm{K}_{\widehat{A}/U}$. Since $\mathrm{K}_{\widehat{A}/U}$ is Cohen--Macaulay, so is $\mathrm{K}_{\widehat{A}}$.
\end{proof}

\begin{cor}\label{cor4.3}
Let $A$ be a Vasconcelos ring of dimension $d \geq 1$.  If $x$ is a nonzerodivisor in $A$, then $A/xA$ is a Vasconcelos ring.
\end{cor}

\begin{proof}
We may assume that $A$ is complete and that $d\geq 2$. Let $U = U_A(0)$  and $C = A/U$.
By Theorem~\ref{2.5}, $C$ is a $d$--dimensional Cohen--Macaulay ring and $\dim U \leq d-2$. Then $\dim U/xU  \leq d-3$. Since $x$ is $C$--regular, $C/xC$ is Cohen--Macaulay. By Theorem~\ref{2.5}, $A/xA$ is a Vasconcelos ring.
\end{proof}

\begin{ex}\label{2.7}
Let $(R, \n)$ be a Cohen--Macaulay local ring of dimension $d \geq 3$. Let $x$ be a nonzerodivisor of $R$ and let $\a$ be an $R$--ideal such that $\mathrm{ht}(\a) \ge 3$ and $x \not\in \a$. Let $A= R/(xR \cap \a)$.
Then $A$ is a non--Cohen--Macaulay Vasconcelos ring of dimension $d-1 \geq 2$.
\end{ex}

\begin{proof}
 Consider the exact sequence $0 \to R/(\a: x) \to A \to R/xR \to 0$.  Then $A/xA$ is Cohen--Macaulay. Also we have
 \[\dim_A(xA) = \dim R/(\a : x) \leq \dim R/\a \leq d-3 = \dim A-2.\]
 By Theorem~\ref{2.5}, $A$ is a non-Cohen--Macaulay Vasconcelos ring with $\dim A = d-1 \geq 2$.
\end{proof}

\begin{ex}\label{ex2.7}
Let $(R, \n)$ be a Cohen--Macaulay local ring with $\dim R = d \geq 2$. Let $M$ be a finitely generated $R$-module with $\dim_RM \le d-2$. Then the idealization $A= R \ltimes M$ is a Vasconcelos ring with $\dim A = d$. Also we have $\Assh(A) = \Min(A)$. Moreover if $R=k [[x,y]]$ and $M=R/(x,y)^2$, then $A=R \ltimes M$ is not quasi--Buchsbaum.
\end{ex}

\begin{proof}
Let $N =(0) \times M$. Consider the exact sequence $0 \to N  \to A \overset{\varepsilon}{\to} R \to 0$, where
$\varepsilon (x) = r$ for each $x=(r,m) \in A$. Then $A/N \simeq R$ is Cohen--Macaulay. Also $\dim_AN = \dim_RM \le d-2$. By Theorem~\ref{2.5}, $A$ is a Vasconcelos ring with $\dim A = d$. Moreover we have
 \[ \Assh(A) = \{\p \times M \mid \p \in \Assh(R) \} = \Min(A).\]
\end{proof}

We are now interested in the question of how Vasconcelos rings are preserved under flat base changes.

\begin{thm}\label{thm4.6}
Let $(A, \m)$ and $(B, \n)$ be Noetherian local rings and let $\varphi : A \to B$ be a flat local homomorphism.
Then the following assertions hold true.
\begin{enumerate}
\item[{\rm (a)}] Suppose that $A$ is a Vasconcelos ring and $B/\m B$ is a Cohen--Macaulay ring. Then $B$ is a Vasconcelos ring.
\item[{\rm (b)}] Suppose that $B$ is a homomorphic image of a Cohen--Macaulay ring. If $B$ is a Vasconcelos ring and  $\Ass_B (B/\p B) = \Assh_B(B/\p B)$ for every $\p \in \Assh(A)$, then
 $A$ is a Vasconcelos ring and $B/\m B$ is a Cohen--Macaulay ring.
\end{enumerate}
\end{thm}

\begin{proof}
Let $U=U_A(0)$.
Notice that $(A/U) \otimes_A B$ is Cohen--Macaulay if and only if both $A/U$ and $B/\m B$ are Cohen--Macaulay.

\medskip

\noindent (a)
We may assume that $A$ is complete and $\dim A=d>0$. Since $A$ is a Vasconcelos ring, by Theorem~\ref{2.5} $A/U$ is Cohen--Macaulay and $\dim_A U \leq d-2$. Since $B/\m B$ is also Cohen--Macaulay by assumption, $B/UB \simeq (A/U) \otimes_A B$ is Cohen--Macaulay. Moreover we have
\[ \dim_B (UB) = \dim_B (U \otimes_A B) = \dim_A U + \dim_B (B/\m B) \leq d-2 +   \dim_B (B/\m B) = \dim B-2.\]
Hence by Proposition~\ref{1.2-1}, $B$ is a Vasconcelos ring.

\medskip

\noindent (b) Suppose that $A$ is Cohen--Macaulay. It is enough to show that $B$ is Cohen--Macaulay.
Let $\P \in \Ass(B)$ and $\p = \P \cap A$.
Then $\p \in \Ass(A) = \Assh(A)$. We have
 $\P \in \Ass(B/\p B) = \Assh(B/\p B)$ by assumption. This means that
 \[
 \dim B/\P = \dim B/\p B = \dim_A (A/ \p) + \dim_B (B/\m B) = d+ \dim_B (B/\m B)= \dim B.
 \]
Therefore $ \Ass(B) = \Assh(B)$, which shows that $B$ is an unmixed Vasconcelos ring.
By Proposition~\ref{Vring} (b), $B$ is Cohen--Macaulay.

Now suppose that $A$ is not Cohen--Macaulay. We claim that $UB=U_B(0)$, the unmixed component of $(0)$ in $B$. Let $\P \in \Ass(B/ UB)$. Then $\P \in \Ass(B/\p B)$ for some $\p \in \Ass(A/U)=\Assh(A)$. By assumption, we have $\P \in \Assh(B/\p B)$. This means that $\dim B/\P = \dim B/\p B = \dim B$. Hence $\Ass(B/UB)= \Assh(B)$, which proves the claim.
Now since $B$ is a Vasconcelos ring, by Theorem~\ref{1.2-2} we have that $B/UB$ is Cohen--Macaulay and $\dim_B UB \leq \dim B-2$. Therefore both $A/U$ and $B/\m B$ are Cohen--Macaulay because $(A/U) \otimes_A B \simeq B/UB$ is Cohen--Macaulay. Also $\dim_B UB \leq \dim B-2$ implies that $\dim U \leq d-2$. In particular, $A$ is a Vasconcelos ring by Proposition~\ref{1.2-1}.
\end{proof}

\begin{cor}
Let $(A,\m)$ be a Vasconcelos ring and let $R = A[X_1, X_2, \cdots, X_n]$ be the polynomial ring. Then $R_P$ is a Vasconcelos ring for every $P \in \mathrm{V}(\m R)$.
\end{cor}

\begin{cor} Let $(A,\m)$ and  $(B,\n)$ be Noetherian local rings which are homomorphic images of Cohen--Macaulay rings. Let $\varphi : A \to B$ be a flat local homomorphism. Then the following two conditions are equivalent{\rm :}
\begin{enumerate}
\item[{\rm (a)}]  $A$ is a Vasconcelos ring and $B/\m B$ is a Cohen--Macaulay ring{\rm ;}
\item[{\rm (b)}] $B$ is a Vasconcelos ring and $\operatorname{Ass}_B(B/\p B) = \operatorname{Assh}_B(B/\p B)$ for every $\p \in \operatorname{Assh} (A)$.
\end{enumerate}
\end{cor}

\begin{proof}
(a) $\Rightarrow$ (b) Let $\p \in \operatorname{Assh}(A)$. Then, since $\p \in \operatorname{Ass}(A/U)$, we have that $\operatorname{Ass}_B (B/\p B) \subseteq \operatorname{Ass}_B (B \otimes_AA/U)$. Hence $\operatorname{dim} B/\P = \operatorname{dim} B$ for every $\P \in \operatorname{Ass}_B(B/\p B)$, since $B \otimes_AA/U$ is a Cohen--Macaulay ring with $\operatorname{dim} (B\otimes_AA/U) = \operatorname{dim} B$.
\end{proof}

Next we show that quasi-unmixed Vasconcelos rings behave well under localization.

\begin{prop}
Let $A$ be a Noetherian local ring of dimension $d$. Suppose that $A$ is a homomorphic image of a Cohen--Macaulay ring and $\operatorname{Assh} (A) = \operatorname{Min} (A)$. If $A$ is a Vasconcelos ring, then $A_\p$ is  a Vasconcelos ring for every $\p \in \operatorname{Spec} A$.
\end{prop}

\begin{proof}
We may assume that $A$ is not a Cohen--Macaulay ring. Let $U = U_A(0)$. Then $U \ne (0)$ by Theorem \ref{1.2-2}. Let $\p \in \operatorname{Spec} A$. Notice that $U\subseteq \p$. Since $(A/U)_\p$ is a Cohen--Macaulay ring we may assume that $U_\p \ne (0)$. Let $\a = (0) : U$. Then $\a \ne A$ and $\operatorname{ht}_A\a \ge 2$, since $\dim U \le d-2$ and $A$ is catenary and equidimensional. Hence $\operatorname{ht}_{A_\p}\a A_\p \ge 2$ and $\operatorname{dim} A_\p / \a A_\p  \le \operatorname{dim}A_\p - 2.$ Therefore $A_\p$ is a Vasconcelos ring by Theorem \ref{1.2-2}.
 \end{proof}

\begin{cor}
Suppose that $A$ is a Vasconcelos ring and $\operatorname{Assh}(\widehat{A}) = \operatorname{Min}(\widehat{A})$.  Then $\widehat{A}_\p$ is a Vasconcelos ring for every $\p \in \operatorname{Spec}\widehat{A}$.
\end{cor}

Suppose that $d > 0$ and let $Q$ be a parameter ideal in $A$. We
denote by $R = \mathcal{R} (Q)= A[Qt]$ (resp. $G = \mathrm{G}(Q))$  the Rees
algebra (resp. the associated graded ring) of $Q$. Let $\mathfrak{M} = \m R + R_+$ be the graded maximal ideal in $R$.

\begin{thm}
Let $A$ be a Noetherian local ring with dimension $d>0$. With the above notation we have the following.
\begin{enumerate}
\item[{\rm (a)}] $A$ is a Vasconcelos ring if and only if  $G_\mathfrak{M}$ is a Vasconcelos ring.
\item[{\rm (b)}] Suppose that $A$ is a homomorphic image of a Cohen--Macaulay ring. If $A$ is a Vasconcelos ring, then $R_\mathfrak{M}$ is a Vasconcelos ring.
\end{enumerate}
\end{thm}

\begin{proof}
(a) Let $Q = (a_1, a_2, \cdots, a_d)$ be a parameter ideal in $A$, and let $f_i = a_i t$ for each $1 \le i \le d$. Then $G_+ = (f_1, f_2, \cdots, f_d)G$ and $f_1, f_2, \cdots, f_d$ forms a linear system of parameters in the graded ring $G$. We furthermore have $$\ell_G(G/[(f_1, f_2, \cdots, f_d)G]^{n+1}) = \ell_A(A/Q^{n+1})$$ for all $n \ge 0$. Hence $e_1(Q) = e_1((f_1, f_2, \cdots, f_d)G_\mathfrak{M})$ and the conclusion follows from Theorem \ref{2.5}.

(b) We may assume that $A$ is not a Cohen--Macaulay ring. Let $U =
U_A(0)$ and $B = A/U$. Then by Theorem \ref{1.2-2} $B$ is a Cohen--Macaulay ring and $\operatorname{dim}U \le d-2$. Consider the canonical epimorphism $\varphi : R \to
\mathcal{R}(QB)$ and let $K = \Ker~\varphi$. Then $K \subseteq
U A[t]$.
Let $\a = (0) : U$.  Let $P \in \operatorname{Ass}_R(K)$ and let $\p = P \cap A$. Then $\a \subseteq \p$, since $\a K = (0)$. Notice that $P$ is the kernel of the canonical epimorphism $\phi : R \to \mathcal{R}([Q + \p]/\p)$.
If $Q \subseteq \p$, then $\p = \m$ and $\operatorname{dim} R/P = \operatorname{dim} \mathcal{R}([Q+\p]/\p) = 0$. If $Q \not\subseteq \p$ we have $\operatorname{dim} R/P = \operatorname{dim} A/\p + 1 \le \operatorname{dim} A/\a + 1 \le d-1$. Therefore $\operatorname{dim} R/P \le d-1$ for every $P \in \operatorname{Ass}_R(K)$, which  shows that $\operatorname{dim}_RK \le d-1 = (d+1) - 2$. Since $\mathcal{R}(QB)$ is a Cohen--Macaulay ring, we have that $R_\mathfrak{M}$ is a Vasconcelos ring by Theorem \ref{1.2-2}.
\end{proof}

We close this section with an application to sequentially Cohen--Macaulay rings. We refer the reader to \cite{GHS} for the definition and details. See also \cite{CC}, \cite{Sch2}, \cite{St}. We use the following characterization.

\begin{prop}[\cite{Sch2}]\label{seqCM2}
Let $A$ be a Noetherian local ring.
Then M is a sequentially Cohen--Macaulay $A$-module if and only if
$M$ admits a Cohen--Macaulay filtration, that is,
a family $\mathcal{M}=\{M_i\}_{0\leq i\leq t}$ $(t>0)$
of   $A$-submodules  of $M$ with
$$M_0=(0)\subsetneq M_1 \subsetneq M_2 \subsetneq \cdots
\subsetneq M_t=M$$ such that
\begin{enumerate}
\item[$(\mathrm{i})$] $\dim_AM_{i-1}<\dim_AM_i$ and
\item[$(\mathrm{ii})$] $M_i/M_{i-1}$ is a Cohen--Macaulay $A$-module
for all $1\leq i\leq t$.
\end{enumerate}
\end{prop}

The result below follows from Proposition~\ref{1.2-1}.

\begin{cor}\label{seqCM}
Let $A$ be a Noetherian local ring with dimension $d>0$. Suppose that $A$ is a sequentially Cohen--Macaulay ring. If $\operatorname{dim} A/\p \ne d-1$ for every $\p \in \operatorname{Ass} (A)$,
then $A$ is a Vasconcelos ring.
\end{cor}

\begin{proof} We may assume that $A$ is not a Cohen--Macaulay ring. Let $U=U_A(0)$. In the notation of Proposition~\ref{seqCM2}, we have that $t\geq 2$ and $M_{t-1}=U$ (\cite[Theorem 2.3]{GHS}). Therefore $\dim U \le d-2$, since $\operatorname{dim} A/\p \ne d-1$ for all $\p \in \operatorname{Ass} (A)$. Thus $A$ is a Vasconcelos ring by Proposition \ref{1.2-1}, because the ring $A/U$ is Cohen--Macaulay.
\end{proof}

Last we show that the
converse of Corollary \ref{seqCM} holds true, when $\operatorname{dim} A = 3$ and $A$ is a homomorphic image of a Cohen--Macaulay ring.

\begin{prop}
Let $A$ be a Noetherian local ring of dimension $3$, which is a homomorphic image of a Cohen--Macaulay ring. If $A$ is a Vasconcelos ring, then $A$ is a sequentially Cohen--Macaulay ring with $\operatorname{dim} A/\p \ne 2$ for every $\p \in \operatorname{Ass} (A)$.
\end{prop}

\begin{proof}
Let $U = U_A(0)$. We may assume that $U \ne (0)$. Then $A/U$ is a Cohen--Macaulay ring and $\dim U \le 1$ by Theorem \ref{1.2-2}. Hence  $\operatorname{dim} A/\p \ne 2$ for all  $\p \in \operatorname{Ass} (A)$. Let $W = \H_\m^0(A)$. Then $W \subseteq U$ and $W = \H_\m^0(U)$, because $\operatorname{depth} (A/U) >0 $. Therefore, if $\dim U = 1$ and $\operatorname{depth}(A) = 1$, then $(0) \subsetneq U \subsetneq A$ is a Cohen--Macaulay filtration of $A$. If $\dim U = 1$ but $\operatorname{depth} (A) = 0$, then $(0) \subsetneq W \subsetneq U \subsetneq A$ is a Cohen--Macaulay filtration of $A$. If $\dim U = 0$, then $(0) \subsetneq U \subsetneq A$ is a Cohen--Macaulay filtration of $A$. Thus $A$ is a sequentially Cohen--Macaulay ring by Proposition \ref{seqCM2}.
\end{proof}

\begin{cor}
Suppose that $A$ is a Vasconcelos ring of dimension $3$. Then the completion $\widehat{A}$ of $A$
  is a sequentially Cohen--Macaulay ring.
\end{cor}

\section{Rings with $e_1(Q)$ constant}

In this section we study the
problem of when $e_1(Q)$ is independent of the choice of parameter
ideals $Q$ in $A$. Part of the motivation comes from the fact that
Buchsbaum rings have this property. We establish here that when
$e_1(Q)=-1$ or $e_1(Q)=-2$ and $A$ is unmixed, then $A$ is indeed
Buchsbaum. (The question of the variability of $e_1(Q)$ will be
considered in another paper.)

Let $(A,\m)$ be a Noetherian local ring with maximal ideal $\m$ and $d=\operatorname{dim}A > 0$. Assume that $A$ is a homomorphic image of a Gorenstein ring. Then $A$ contains a system of parameters $x_1, x_2, \cdots, x_d$ which forms a strong $d$-sequence in $A$, that is, the sequence $x_1^{n_1}, x_2^{n_2}, \cdots, x_d^{n_d}$ is a $d$-sequence in $A$ for all integers $n_1, n_2, \cdots, n_d \ge 1$ (see \cite[Theorem 2.6]{Cu} or \cite[Theorem 4.2]{Kw} for the existence of such systems of parameters). For each integer $q \ge 1$
let $\Lambda_q(A)$ be the set of values $e_1(Q)$, where $Q$ runs over the parameter ideals of $A$ such that $Q \subseteq \m^q$ and $Q=(a_1, a_2, \cdots, a_d)$ with $a_1, a_2, \cdots, a_d$ a $d$-sequence. We then have $\Lambda_q(A) \ne \emptyset$, $\Lambda_{q+1}(A) \subseteq \Lambda_q(A)$ for all $q \ge 1$, and $\alpha \le 0$ for every $\alpha \in \Lambda_q(A)$ (Corollary \ref{cor} (a)).

\begin{lem}\label{key}
Let $(A,\m)$ be a Noetherian local ring of dimension $d\ge 2$, which is a homomorphic image of a Gorenstein ring. Assume that $\Lambda_q(A)$ is a finite set for some integer $q \ge 1$ and let $\ell = -\operatorname{min}\Lambda_q(A)$. Suppose that $\operatorname{Ass} (A)= \operatorname{Assh} (A)$. Then $\m^{\ell}\H_\m^i(A) = (0)$ for all $i \ne d$, whence all the local cohomology modules $\{\H_\m^i(A)\}_{0 \le i < d}$ of $A$ are finitely generated.
\end{lem}

\begin{proof}
Let $C = \H_\m^1(A)$. Then $C$ is a finitely generated $A$-module (Proposition \ref{GNa} (a)). Suppose that $d = 2$ and let $\ell' = \ell_A(C)$. Let $a, b$ be a system of parameters of $A$ and assume that $a,b$ is a $d$-sequence in $A$. Then the element $a$ is superficial for  the ideal $Q = (a,b)$, so that $e_1(Q) = e_1(Q/(a)) = -\ell_A((0):_C a )$. Therefore, choosing $a, b \in \m^q$ with $aC = (0)$, we get $-\ell' =e_1(Q) \in \Lambda_q(A)$, whence $\ell' \le \ell$. Thus $\m^{\ell}C = (0)$, because $\m^{\ell'}C = (0)$.

Suppose now that $d \ge 3$ and that our assertion holds true for $d-1$. Let $$\mathcal{F}' = \{\p \in \operatorname{Spec} A \mid \p \ne \m, \operatorname{dim}A_\p > \operatorname{depth}(A_\p) = 1\}.$$ Then $\mathcal{F}'$ is a finite set (Proposition \ref{GNa} (b)). We choose $x \in \m$ so that $$x \not\in \bigcup_{\p \in \operatorname{Ass}(A)}\p \cup \bigcup_{\p \in \mathcal{F}'} \p.$$ Let $n \ge q$ be an integer such that $x^n\H_\m^1(A) = (0)$ and put $y = x^n$. Let $B =A/yA$. Then $\operatorname{dim} B = d-1$ and $\operatorname{Ass}_A  (B) \setminus \{\m \} = \operatorname{Assh}_A(B)$. It follows that $U_B(0) = \H_\m^0(B)$ (see Remark~\ref{U(A)}). Let $\widetilde{B}=B/\H_\m^0(B)$.

Let $q' \ge q$ be an integer such that ${\n}^{q'} \cap \H_\m^0(B) = (0)$, where $\n$ denotes the maximal ideal of $B$. Let $y_2, y_3, \cdots, y_d \in \m^{q'}$ be a system of parameters for the $A$-module $\widetilde{B}$ and assume that $y_2, y_3, \cdots, y_d$ is a $d$-sequence in $\widetilde{B}$. Since $(y_2, y_3, \cdots, y_d)B \cap \H_\m^0(B) = (0)$, we have that $y_2, y_3, \cdots, y_d$ forms a $d$-sequence in $B$ also. Then, because $y$ is $A$-regular, the sequence $y_1=y, y_2, \cdots, y_d$ forms a $d$-sequence in $A$, whence $y_1$ is superficial for the parameter ideal $Q = (y_1, y_2, \cdots, y_d)$ of $A$. Consequently
$$e_1((y_2, y_3, \cdots, y_d)\widetilde{B})=e_1((y_2, y_3, \cdots, y_d)B) = e_1(Q) \in \Lambda_q(A),$$ so that $\Lambda_{q'}(\widetilde{B}) \subseteq  \Lambda_q(A).$
Therefore $\Lambda_{q'}(\widetilde{B})$ is a finite set, whence the hypothesis of induction on $d$ yields that $\m^{\ell''}\H_\m^i(\widetilde{B}) = (0)$ for all $i \ne d-1$, where $\ell'' = -\operatorname{min} \Lambda_{q'}(\widetilde{B})$.
Hence $\m^{\ell}\H_\m^i(\widetilde{B}) = (0)$ for all $i \ne d-1$, because $\ell'' \le \ell$.

Now consider the exact sequence
$$\cdots \to \H_\m^1(A) \overset{x^n}{\to} \H_\m^1(A) \to  \H_\m^1(B) \to \cdots \to \H_\m^i(B) \to \H_\m^{i+1}(A) \overset{x^n}{\to} \H_\m^{i+1}(A) \to \cdots $$
of local cohomology modules. We then have
$$
\m^{\ell}\left[(0):_{\H_\m^{i+1}(A)}x^n\right] = (0)
$$
for all integers $1 \le i \le d-2$ and $n \ge q$, because $\m^{\ell}\H_\m^i(B) = (0)$ for all $1 \le i \le d - 2$.  Hence $\m^{\ell}\H_\m^{i+1}(A)= (0)$, because $$\H_\m^{i+1}(A) = \bigcup_{n \ge 1}\left[(0):_{\H_\m^{i+1}(A)}\m^n\right].$$
On the other hand we  have the embedding $\H_\m^1(A) \subseteq  \H_\m^1(B)$, since $x^n\H_\m^1(A) = (0)$. Thus
$\m^{\ell}\H_\m^{1}(A) = (0)$, which completes the proof of the lemma.
\end{proof}

Now let $A$ be a Noetherian local ring with maximal ideal $\m$ and $d = \operatorname{dim}A > 0$. Let $\Lambda (A) = \{e_1(Q) \mid Q ~\operatorname{is ~a ~parameter ~ideal ~in} A \}$. Passing to the completion $\widehat{A}$ of $A$ and applying Lemma~\ref{key} we obtain the following.

\begin{prop}\label{flc}
Let $(A,\m)$ be an unmixed Noetherian local ring of dimension $d\ge 2$. Assume that $\Lambda (A)$ is a finite set and put $\ell = -\operatorname{min} \Lambda (A)$. Then $\m^{\ell}\H_\m^i(A) = (0)$ for every $ i \ne d$, whence $\H_\m^i(A)$ is a finitely generated $A$-module for every $i \ne d$.
\end{prop}


A system of parameters $a_1, a_2, \cdots, a_d$ of $A$ is said to be {\it standard}, if it forms a $d^+$-sequence, that is, $a_1, a_2, \cdots, a_d$ forms a strong $d$-sequence in any order. We have that $A$ possesses a standard system of parameters if and only if $A$ is a generalized Cohen--Macaulay ring, i.e., all the local cohomology modules $\{\H_\m^i(A) \}_{0 \le i < d}$ are finitely generated (see \cite{T}).

We say that a parameter ideal $Q$ of $A$ is standard, if it is
generated by a standard system of parameters. We have that $Q$ is
standard if and only if the equality $$\ell_A(A/Q) - e_0(Q) =
\sum_{i=0}^{d-1}\binom{d-1}{i}h^i(A):={\mathbb I}(A)$$ holds true,
where $h^i(A) = \ell_A(\H_\m^i(A))$ for each $i \in \mathbb Z$ (cf. \cite[Theorem 2.1]{T}). See \cite{STC, T} for details, where the notion of generalized Cohen--Macaulay module is also given and various basic properties of generalized Cohen--Macaulay rings and modules are explored.

\medskip

Assume that $A$ is a generalized Cohen--Macaulay ring with $d \ge 2$ and let $$s=\sum_{i=1}^{d-1}\binom{d-2}{i-1}h^i(A).$$ If $Q$ is a parameter ideal of $A$, by \cite[Lemma 2.4]{GNi} we have that
$$
e_1(Q) \ge -s,
$$
where the equality holds true, if $Q$ is standard (\cite[Korollar 3.2]{Sch1}).

Therefore, if $A$ is unmixed, $d \ge 2$, and $\Lambda (A)$ is a finite set, by Proposition~\ref{flc} we have that
$$
\m^s \H_\m^i(A)= (0)
$$
for all $i \ne d$. This exponent is, however, never the best possible, as we show in the following.

\begin{cor}\label{q-Bbm}
Let $(A,\m)$ be an unmixed Noetherian local ring of dimension $d\ge 2$. If $\# \Lambda (A) = 1$, then $A$ is a quasi-Buchsbaum ring, that is, $\m \H_\m^i(A) = (0)$ for all $i \ne d$.
\end{cor}

To prove Corollary \ref{q-Bbm} we need the lemmas below.

\begin{lem}\label{d=2}
Suppose that $(A,\m)$ is a Noetherian local ring of dimension $d= 2$ and $\operatorname{depth} (A) = 1$. Assume that $\H_\m^1(A)$ is finitely generated. Let $Q$ be  a parameter ideal of $A$. Then the following three conditions are equivalent{\rm :}
\begin{enumerate}
\item[{\rm (a)}] $e_1(Q) = -\ell_A(\H_\m^1(A))${\rm ;}
\item[{\rm (b)}] $Q \H_\m^1(A) = (0)${\rm ;}
\item[{\rm (c)}] $Q$ is standard.
\end{enumerate}
\end{lem}

\begin{proof} (c) $\Rightarrow$ (a) See \cite[Lemma 2.4 (2)]{GNi}.

(b) $\Leftrightarrow$ (c)  See \cite[Corollary 3.7]{T}.

(a) $\Rightarrow$ (b) We may assume that the field $A/\m$ is infinite. Let $Q = (a,b)$ be such that each one of $a,b$ is a superficial element of $Q$. Then $$-\ell_A(\H_\m^1(A)) = e_1(Q) = e_1(Q/(a)) = -\ell_A(\H_\m^0(A/(a))$$ by \cite[Lemma 2.1 (1)]{GNi}. Hence $\ell_A(\H_\m^1(A))=\ell_A((0) :_{\H_\m^1(A)} a)$, and so $a\H_\m^1(A) = (0)$. Similarly we get $b\H_\m^1(A) = (0)$, whence $Q\H_\m^1(A) = (0)$.
\end{proof}

\begin{lem}\label{d=2generalized}
Suppose that $(A,\m)$ is a generalized Cohen--Macaulay local ring of dimension $d \ge 2$ and $\operatorname{depth} (A) > 0$. Let $Q$ be a parameter ideal of $A$ such that $e_1(Q)= - \sum_{i = 1}^{d-1}\binom{d - 2}{i - 1}h^i(A)$. Then $Q \H_\m^i(A) = (0)$ for all $i \ne d$.
\end{lem}

\begin{proof}
If $d=2$ the conclusion follows from Lemma~\ref{d=2}. Assume that $d\geq 3$. Let $Q = (a_1, a_2, \cdots, a_d)$, where each $a_i$ is superficial for the ideal $Q$, and let $a = a_i$. Let $B=A/aA$. We have that $e_1(QB)=e_1(Q)$.

Consider the exact sequence of local cohomology modules $$
\cdots \to \H_\m^i(A) \overset{a}{\to} \H_\m^i(A) \to \H_\m^i(B) \to \H_\m^{i+1}(A) \overset{a}{\to}  \H_\m^{i+1}(A) \to \H_\m^{i+1}(B) \to \cdots.
$$ We then have
$$
h^i(B) = \ell_A(\H_\m^i(A)/a\H_\m^i(A)) + \ell_A((0) :_{\H_\m^{i+1}(A)}a)
\le h^i(A) + h^{i+1}(A)
$$
for all $0 \le i \le d-2$. Hence
\begin{eqnarray*}
 e_1(QB)\geq -\sum_{i=1}^{d-2}\binom{d-3}{i-1}h^i(B) &\ge&-\sum_{i=1}^{d-2}\binom{d-3}{i-1}\left[h^i(A) + h^{i+1}(A)\right]\\
 &=& -\sum_{i=1}^{d-1}\binom{d-2}{i-1}h^i(A)\\
&=& e_1(QB).
\end{eqnarray*}
It follows that $h^i(B) = h^i(A) + h^{i+1}(A)$ for every $1 \le i \le d-2$, whence $a\H_\m^i(A) = (0)$ for all $1 \le i \le d - 1$. Thus $Q\H_\m^i(A) = (0)$, if $i \ne d$.
\end{proof}

\begin{proof}[Proof of Corollary $\ref{q-Bbm}$]
By Proposition \ref{flc} $A$ is a generalized Cohen--Macaulay ring. Hence we have $\Lambda (A) = \{-\sum_{i=1}^{d-1}\binom{d-2}{i-1}h^i(A)\}$ by \cite[Korollar 3.2]{Sch1}.
Let $a \in \m$ such that $\operatorname{dim} A/aA = d-1$. It is enough to show that $a\H_\m^i(A) = (0)$ for all $i \ne d$. This follows from Lemma~\ref{d=2generalized}.
\end{proof}

Since every quasi-Buchsbaum ring $A$ is Buchsbaum once $\operatorname{depth} A \ge d-1$ (\cite[Corollary 1.1]{SV}), we readily get the following.

\begin{cor} Suppose that $A$ is an unmixed Noetherian local ring of dimension $d \ge 2$ and $\operatorname{depth} A \ge d-1$. Then $\# \Lambda (A) = 1$ if and only if $A$ is a Buchsbaum ring.
\end{cor}

The authors expect that $A$ is a Buchsbaum ring, if $A$ is unmixed and $e_1(Q)$ is independent of the choice of parameter ideals $Q$ of $A$. We will show that this is the case, when $e_1(Q) = -1$ and when $e_1(Q) = -2$.

\begin{prop}\label{dimU}
Let $A$ be a Noetherian local ring of dimension $d \ge 2$ and suppose that for all parameter ideals $Q$ of $A$, $e_1(Q)=-t$ for some $t\geq 0$. Let $U=U_A(0)$. Then $\dim U \le d-2$, and for all parameter ideals $\q$ of $A/U$ we have that $e_1(\q)=-t$ .
\end{prop}

\begin{proof}
Let $B=A/U$. Assume that $\dim U = d-1$. Choose a system of parameters $a_1, a_2, \cdots, a_d$ of $A$ so that $(a_d) \cap U = (0)$ (cf. \cite{CC}). Let $\ell > t$ be an integer and let $Q = (a_1^\ell, a_2, \cdots, a_d)$. For all $n \ge 0$ we have the exact sequence of $A$-modules
$$0 \to U/(Q^{n+1} \cap U) \to A/Q^{n+1} \to B/Q^{n+1}B \to 0.$$ Let $k \ge 0$ be an integer such that
$$Q^{n} \cap U =
Q^{n-k}(Q^k \cap U)$$
for all $n \ge k$. We put $U' = Q^k \cap U$ and $\q = (a_1^{\ell}, a_2, \cdots, a_{d-1})$. Then $Q^{n-k}U' = \q^{n-k}U'$, because $a_dU=(0)$. Therefore, for all $n \ge k$ we have
$$\ell_A(A/Q^{n+1}) = \ell_A(B/Q^{n+1}B) + \ell_A(U'/\q^{n-k+1}U') + \ell_A(U/U'),$$
which implies
$$
-t = e_1(Q) = e_1(QB) -e_0(\q U').
$$
Consequently, since $e_1(QB) \le 0$, we have $$
\ell \le \ell e_0((a_1, a_2, \cdots, a_{d-1})U') = e_0(\q U') = e_1(QB) + t \le t,
$$
which is impossible. Thus $\dim U \le d-2$.

To see the second assertion, let $\q$ be a parameter ideal of $B$. Then, choosing a parameter ideal $Q$ of $A$ so that $QB = \q$, we get $e_1(\q) = e_1(Q)=-t$, since $\dim U \le d-2$.
\end{proof}

\begin{thm}\label{constant}
Let $(A,\m)$ be a Noetherian local ring of dimension $d \ge 2$. Then the following two conditions are equivalent{\rm :}
\begin{enumerate}
\item[{\rm (a)}] $e_1(Q) = -1$ for every parameter ideal $Q$ of $A${\rm ;}
\item[{\rm (b)}] Let $U = U_{\widehat{A}}(0)$ be the unmixed component of $(0)$ in the $\m$-adic completion $\widehat{A}$ of $A$. Then $\operatorname{dim}_{\widehat{A}}U \le d-2$ and $\widehat{A}/U$ is a Buchsbaum ring such that either
\begin{enumerate}
\item[$(\mathrm{i})$] $\H_{\widehat{\m}}^i(\widehat{A}/U) = (0)$ for all $i \ne 1, d$ and $\ell_{\widehat{A}}(\H_{\widehat{\m}}^1(\widehat{A}/U)) = 1$, or
\item[$(\mathrm{ii})$] $\H_{\widehat{\m}}^i(\widehat{A}/U) = (0)$ for all $i \ne d-1, d$ and $\ell_{\widehat{A}}(\H_{\widehat{\m}}^{d-1}(\widehat{A}/U)) = 1$,
\end{enumerate}
where $\widehat{\m}$ denotes the maximal ideal of $\widehat{A}$.
\end{enumerate}
\end{thm}

\begin{proof}
We may assume that $A$ is complete.

(b) $\Rightarrow$ (a)
Let $Q$ be a parameter ideal of $A$, and let $B = A/U$. Then $e_1(Q) = e_1(QB)$, since $\operatorname{dim}_{A}U \le d-2$. Consequently $e_1(Q) = -1$, because $e_1(QB) = -\sum_{i=1}^{d-1}\binom{d-2}{i-1}h^i(B) = -1$ by \cite[Korollar 3.2]{Sch1} and condition (i) or (ii).

(a) $\Rightarrow$ (b)
By Proposition~\ref{dimU} we may assume that $A$ is unmixed. Then $A$ is a quasi-Buchsbaum ring by Corollary~\ref{q-Bbm}. We have condition (i) or (ii), because $e_1(Q) = -\sum_{i=1}^{d-1}\binom{d-2}{i-1}h^i(A)$. Hence $A$ is a Buchsbaum ring, because $A$ is a quasi-Buchsbaum ring with $\H_\m^i(A) = (0)$ for all $i \ne \operatorname{depth} A, \operatorname{dim} A$ (see \cite[Corollary 1.1]{SV}).
\end{proof}

To treat the case where $e_1(Q) = -2$ we need the following.

\begin{lem}\label{-2}
Suppose that $(A,\m)$ is a generalized Cohen--Macaulay local ring of dimension $d \ge 2$ and $\operatorname{depth} (A) > 0$. Assume that $h^1(A) = 1$ and $h^i(A) = 0$ for all $2 \le i \le d-2$. Let $Q$ be a parameter ideal in $A$ such that $e_1(Q)= - \sum_{i = 1}^{d-1}\binom{d - 2}{i - 1}h^i(A)$. Then $Q$ is standard.
\end{lem}

\begin{proof}
We may assume that the field $A/\m$ is infinite. If $d=2$, the conclusion follows from Lemma~\ref{d=2}.
Let $d \ge 3$ and let $Q = (a_1, a_2, \cdots, a_d)$, where each  $a_i$ is superficial for the ideal $Q$. Let $a = a_i$, $\mathrm{U}(a) = (a) : \m$, and put $B = A/aA$, $\widetilde{B} = B/\H_\m^0(B)$. Then
$$e_1(Q) = e_1(QB) = e_1(Q\widetilde{B}).$$
Consider the exact sequence of local cohomology modules $$
\cdots \to \H_\m^i(A) \overset{a}{\to} \H_\m^i(A) \to \H_\m^i(B) \to \H_\m^{i+1}(A) \overset{a}{\to}  \H_\m^{i+1}(A) \to \H_\m^{i+1}(B) \to \cdots.
$$

 Since $h^1(A) = 1$, we have $h^0(B)=1$, whence  $\H_\m^0(B)=\mathrm{U}(a)/(a)$ and $\widetilde{B} = A/\mathrm{U}(a)$.

By Lemma~\ref{d=2generalized} we have that $a\H_\m^{d-1}(A)=(0)$. Therefore, if $d=3$ we get an exact sequence
$$0 \to \H_\m^1(A) \to \H_\m^1(B) \to \H_\m^2(A) \to 0,$$ whence $h^1(\widetilde{B}) = h^1(B) = 1 + h^2(A)$.

If $d\geq 4$ we get $\H_\m^1(B) \cong \H_\m^1(A)$, $\H_\m^i(B) = (0)$ if $2 \le i \le d-3$, and $\H_\m^{d-2}(B) \cong \H_\m^{d-1}(A)$.

Consequently, if $d\geq 3$ we have that $e_1(Q\widetilde{B}) = - \sum_{i = 1}^{d-2}\binom{d - 3}{i - 1}h^i(\widetilde{B})$, and so by induction on $d$ the parameter ideal $Q\widetilde{B}$ is standard.
 Therefore
$$
\ell_{\widetilde{B}}(\widetilde{B}/Q\widetilde{B}) - e_0(Q\widetilde{B}) = {\mathbb I}(\widetilde{B})=
 (d-2) + h^{d-1}(A).
$$

Now assume by contradiction that $Q$ is not a standard parameter ideal in $A$. Then
\begin{eqnarray*}
(d-2) + h^{d-1}(A) &=& \ell_{\widetilde{B}}(\widetilde{B}/Q\widetilde{B}) -  e_0(Q\widetilde{B})\\ &=& \ell_A(A/(\mathrm{U}(a) + Q)) - e_0(A)\\
&=& \left[\ell_A(A/Q) - e_0(A) \right] - \ell_A((Q + \mathrm{U}(a))/Q)\\
&<&{\mathbb I}(A) - \ell_A((Q + \mathrm{U}(a))/Q)\\
&=& \left[(d-1) + h^{d-1}(A)\right] - \ell_A((Q + \mathrm{U}(a))/Q).
\end{eqnarray*}
Consequently, $\ell_A((Q + \mathrm{U}(a))/Q)=0$, whence $\mathrm{U}(a) \subseteq Q$.
Therefore $$\sum_{i = 1}^d\mathrm{U}(a_i) = Q$$ by the symmetry among the elements $a_i$. Let $\widetilde{A}$ denote the $(\mathrm{S}_2)$-fication of $A$ and look at the canonical exact sequence
$$0 \to A \to \widetilde{A} \to \H_\m^1(A) \to 0$$
(\cite[Theorem 1.6]{AG}). Then $\operatorname{dim} \widetilde{A} = d$, $\operatorname{depth} \widetilde{A} \ge d-1$, and  $\H_\m^{d-1}(\widetilde{A}) \cong \H_\m^{d-1}(A)$. Hence $Q$ is a standard parameter ideal for the generalized Cohen--Macaulay $A$-module $\widetilde{A}$ by \cite[Corollary 3.7]{T}, because $Q\H_\m^{d-1}(\widetilde{A}) = (0)$. Therefore, since $\operatorname{depth}_A\widetilde{A} \ge 2$, any two of $a_1, a_2, \cdots, a_d$ form an $\widetilde{A}$-regular sequence, whence $\mathrm{U}(a_i) \subseteq a_i\widetilde{A}$  for all $1 \le i \le d$. Consequently $\mathrm{U}(a_i) = a_i \widetilde{A}$, because $a_iA \subsetneq \mathrm{U}(a_i) \subseteq a_i \widetilde{A}$ and $\ell_A(a_i \widetilde{A}/a_iA) =\ell_A(\widetilde{A}/A) =  1$. Thus $Q = Q\widetilde{A}$, so that we have $Q^{n+1} = Q^{n+1}\widetilde{A}$ for all $n \ge 0$. Since $$\ell_A(A/Q^{n+1}) = \ell_A(\widetilde{A}/Q^{n+1}\widetilde{A}) - 1,$$ we get $$e_1(Q) = e_1(Q\widetilde{A}) = -h^{d-1}(\widetilde{A}) = -h^{d-1}(A)$$ by \cite[Korollar 3.2]{Sch1}, a contradiction. Thus the parameter ideal $Q$ is standard in $A$.
\end{proof}

\begin{thm}\label{constant2}
Let $(A,\m)$ be a Noetherian local ring of dimension $d \ge 2$. Then the following two conditions are equivalent{\rm :}
\begin{enumerate}
\item[{\rm (a)}] $e_1(Q) = -2$ for every parameter ideal $Q$ of $A${\rm ;}
\item[{\rm (b)}] Let $U = U_{\widehat{A}}(0)$ be the unmixed component of $(0)$ in the $\m$-adic completion $\widehat{A}$ of $A$. Then $\operatorname{dim}_{\widehat{A}}U \le d-2$ and $\widehat{A}/U$ is a Buchsbaum ring with $$\sum_{i = 1}^{d-1}\binom{d-2}{i-1}h^i(\widehat{A}/U) = 2.$$
\end{enumerate}
\end{thm}

\begin{proof}

(b) $\Rightarrow$ (a) The assertion follows from \cite[Korollar 3.2]{Sch1}.

(a) $\Rightarrow$ (b)
By Proposition~\ref{dimU}, we may assume that $A$ is an unmixed complete local ring. Hence $A$ is a quasi-Buchsbaum ring by Corollary \ref{q-Bbm} and $\sum_{i = 1}^{d-1}\binom{d-2}{i-1}h^i(A) = 2$ by \cite[Korollar 3.2]{Sch1}. By \cite[Corollary 1.1]{SV} we may assume that $d \ge 3$, $h^1(A) = h^{d-1}(A) = 1$, and $h^i(A) = 0$ if $2 \le i \le d-2$. Then by Lemma \ref{-2} every parameter ideal $Q$ of $A$ is standard, so that $A$ is a Buchsbaum ring.
\end{proof}

\end{document}